\documentclass[runningheads,a4paper]{llncs}
\usepackage{amssymb}
\usepackage{graphicx}

\usepackage{url}

\newtheorem{thm}{Theorem}

\newtheorem{lem}[thm]{Lemma}

\newtheorem{str}[thm]{Construction}
\mathchardef\hy="2D
\usepackage{latexcad}
\begin{document}

\mainmatter

\title{On building 4-critical plane and projective plane multiwheels from odd wheels}
\titlerunning{On building 4-critical plane and projective plane multiwheels}
\author{Dainis Zeps}
\authorrunning{Dainis Zeps}
\institute{Institute of Mathematics and Computer Science, \\ University of Latvia, \\
Rainis blvd., Riga, Latvia\\
\url{dainize@mii.lu.lv}\\
\url{http://www.ltn.lv/~dainize/}}

\toctitle{On building 4-critical graphs}
\tocauthor{Dainis Zeps}

\maketitle

\begin{abstract}
We build unbounded classes of plane and projective plane multiwheels that are 4-critical that are received summing odd wheels as edge sums modulo two. These classes can be considered as ascending from single common graph that can be received as edge sum modulo two of the octahedron graph $O$ and the minimal wheel $W_3$. All graphs of these classes belong to $2n-2$-edges-class of graphs, among which are those that quadrangulate projective plane, i.e., graphs from Gr\"{o}tzsch class, received applying Mycielski's Construction to odd cycle \cite{bo 08}.
\end{abstract}

\section{Introduction}
We are using terminology from \cite{kra 94,moh 01,ze 98}.

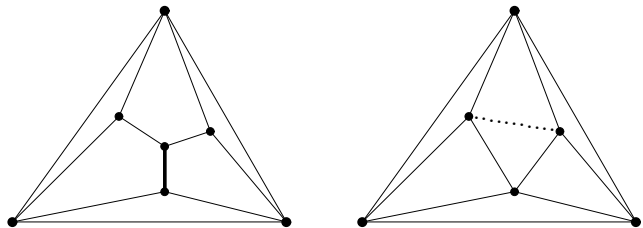
\placedrawing[h]{OCT1.LP}{In the graph $w_{1_3}$ (left) contracting thick edge we get octahedron graph without an edge ($O^-$, right). Adding dotted edge we get octahedron graph $O$.}{fig0}

We consider graph in fig.\ref{fig0} left, that in \cite{ze 11} is denoted $G_3$, but in this article gets several denotations due to its particular features, $w_{111}$ or $w_{1_3}$, $g_{111}$ or $g_{1_3}$, see lower. In this introduction we consider some simple features of this  graph $w_{1_3}$ called {\em base graph} in frames of this article, but further in the article we generalize these features to two, plane and projective plane, unbounded classes of graphs.

We start from a simple observation that contraction of one of edges (incident to degree 3 vertices) in the graph ($w_{1_3}$) turns it into octahedron graph ($O$) minus an edge ($O^-$), see fig.\ref{fig0}. Now it is obvious that the graph $O^-$ is a minor of $w_{1_3}$, but the graph $O$ is not.

Further, we may express this fact  using terminology from \cite{ze 11}, i.e., $<O^-,O>$ is minor bracket for the graph $w_{1_3}$. We remind that $<h_1,h_2>$ is minor bracket for $G$ if $h_1\prec h_2$, $h_1 \prec G$ and $h_2 \not \prec G$, \cite{ze 11}. The fact expressed as predicate $<h_1,h_2:G>$.

One more observation that the base graph $w_{1_3}$ turns into the wheel graph $W_3$ after contracting three edges incident to three and four degree vertices.

Further, graph $w_{1_3}$ is 4-critical \cite{ze 11}, and this fact might be verified directly for such a small graph. But for further discussion we need to examine this graph more closely.

The base graph $w_{1_3}$ can be considered as edge sum modulo two of three graphs $W_3$ in the way that each pair overlap just in one edge, and all three wheels have one vertex in common, see fig. \ref{fig6}. Let us assume that overlap either two rim edges, or two spike edges, or spike edge and rim edge, all three ways giving the same graph $w_{1_3}$ (see fig. \ref{fig6}): trivial (automorphism) fact, but not so for further, see below. Further we are going to use this consideration of $w_{1_3}$ as sum of three wheels, and for that reason we use this multi-index denotation for it, i.e., $w_{111}$ (or $w_{1_3}$) with three indices (three ones), not single index, and second index as factor of equal indices, see lower.  Further we are going to use by edge summation modulo two denotations with  simple arithmetic operators, i.e., multiplication and sum operators, e.g. $w_{1_3} = 3 W_3$. Of course, corresponding graph operations behind are indeterministic, i.e., depending on how we configure graphs one against other by edge summation modulo two, i.e., which elements of corresponding graphs we allow to overlap. Besides, $w_{1_3}$ may be expressed in the way $O+W_3$ with rim edges annihilating with a triangle of octahedron graph, thus having one more equation $w_{1_3} = O + W_3$ becoming equality under specified conditions.

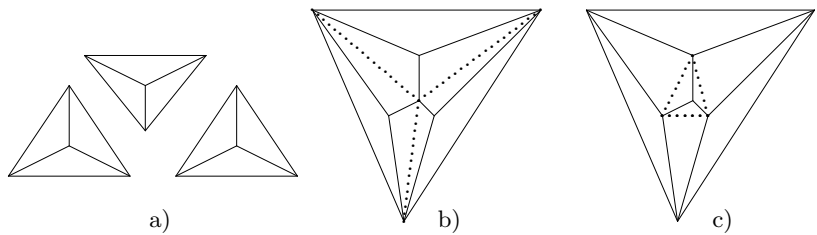
\placedrawing[h]{OCT6.LP}{Forming graph $w_{1_3}$ (b) by summation modulo two of edges of three graphs $W_3$  (a): $3 W_3=w_{1_3}$. Edges that annihilate (one in one pair) by summation modulo two are depicted as dotted (b). In c) $w_{1_3}$ is received by operation $O+W_3$ where triangle (depicted dotted) is annihilated in both graphs.}{fig6}

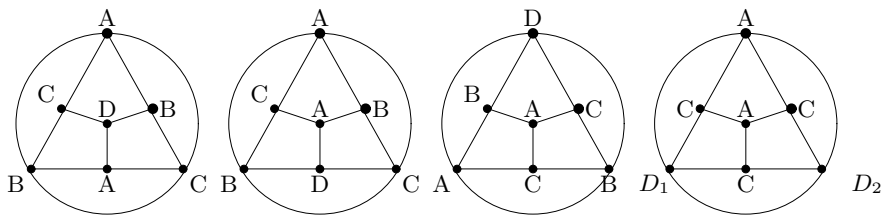
\placedrawing[h]{OCT7.LP}{Four possible colorings of $w_{1_3}$. $D$ designates lonely color. Lonely color $D$ may color central hub (first), section hub (second), rim (third), and two lonely vertices as two rim vertices (fourth).}{fig7}

The graph $w_{1_3}$ can be colored in four distinct ways, see fig.\ref{fig7}. We distinct them in way lonely color, $D$, is applied: the lonely $D$ may color central hub, section hub, rim and two vertices of rim, thus giving four ways. We remind that only in chromatic critical graph every vertex may receive lonely color. Thereby, each vertex for $w_{1_3}$ may be colored with lonely color, thus proving that graph is 4-critical.

Further, the base graph $w_{1_3}$ may be embedded on projective plane, quadrangulating it, see a) in fig. \ref{fig10} lower. Clearly, $w_{1_3}$ belongs to class of graphs with $2n-2$ edges by $n$ vertices, where all graphs quadrangulating projective plane should belong. Besides, $w_{1_3}$ is selfdual, if considered plane, but of course not such on projective plane. It is remarkable that dual graph to octahedron is cube graph, but cube graph with one corner cut off becomes just the graph $w_{1_3}$, that is selfdual (see \cite{ze 11}).

The aim of this article is to show that these facts concerning this single and very simple graph $w_{1_3}$ may be extended for unbounded classes of 4-critical graphs, both planar and projective planar. These classes we are going to build from arbitrary odd wheel graphs.

The theorem about octahedron minor bracket is mentioned in \cite{ze 08}, but here we give proof of this fact.

\section{Defining 4-critical plane  multiwheels}

Let us assume that the wheel $W_k$ is built from $k$ {\em simple sections} where as simple section we take triangle $C_3$ with one vertex common from each triangle  for the wheel's hub, and opposite edge of triangle as for forming wheel's rim. In other words, wheel $W_k$ is the sum of its $k$ sections (triangles) into a simple graph (with double-edges turning into simple edges). By the way, if we applied this same summation of the edges modulo two then the sides of triangles that touch each other would annihilate, and remnant/resulting graph would be simple cycle (rim of the wheel) and isolated vertex (hub of the wheel). As an obvious observation, let us notice that edge sum modulo two of $k>3$ triangles gives $C_k$ and isolated vertex only in case triangles annihilate two edges, and one vertex becomes common for all triangles.

If we take wheel $W_k$ to be odd ($k$ is odd) then we get 4-critical graph. Odd wheels are simplest unbounded class of vertex-3-connected 4-critical graphs. (Let us remember that odd cycles is the only graphs that are 3-critical.) We are going to generalize this class to 4-critical vertex-3-connected plane and projective plane multiwheels.

Let us do similar as higher summation of edge sets of some wheels. Let us replace each section in a simple odd wheel with another arbitrary odd wheel in a way that edge sets of sections are summed modulo two.  This new aggregation of wheels $M$ may be expressed as $\sum W_{k_i}$ where summation is modulo two over index $i$ numbering sections that were replaced by wheels of order $k_i$ in each case. Evidently the summation itself is indeterministic because result of it depends on how wheels overlap each other by summation. Now we ask: under which conditions the resulting graph is 4-critical. It turns out that the answer directly is connected with the number of edges wheels intersect by summation. Theorem \ref{thm1} below says that this number must be equal to two.

But first we are to prove some lemmas  about intersecting wheels without edge losses by summation modulo two. For example, $W_5+W_5$ may be formed with all five rim vertices of both wheels common and forming subgraph $K_5$, but resulting 6-chromatic graph is in no way chromatic critical.

\begin{lem}\label{ov} Let  $k$ odd wheels by summation possibly intersect in vertices but not in edges. The sum of these wheels can't give 4-critical graph except in case $k=1$.
\end{lem}
\begin{proof}Let by summation of wheels $k>1$ wheels  have become vertex connected components, that are wheels all the same, in the resulting graph. It is obvious that elimination of edge or vertex, or edge contraction in one wheel can't affect coloring  of others in four colors in the resulting graph. Thus, the graph can't be chromatic critical.
\end{proof}

Let us configure two odd wheels so that they have common two adjacent vertices, and sum their edges modulo two. It is easy to see that the resulting graph is 3-chromatic. It suffices to notice that losing of an edge in both odd wheels allow to color them in 3-colors so that lost edge's ends receive the same color. Further, we may easily apply the use of this fact to {\em unclosed sequence of wheels} $W_{q_1}$,...,$W_{q_i}$ where two proximal wheels overlap in two adjacent vertices, but next two possibly only in one. Let us formulate it as a lemma.

\begin{lem}\label{us}  Let summation of edges modulo two is applied to unclosed sequence of wheels. The resulting graph is 3-chromatic.
\end{lem}

We need one more crucial feature of 4-critical graphs. Let graph $H$ is 4-critical and let $H'=H \odot w$ be graph $H$ with vertex $w$ split into two new vertices and edges incident to $w$ be connected either to one or other vertex. We ask whether graph $H'$ can remain to be 4-critical. Of course, it is expectable that $H'$ becomes 3-chromatic.

\begin{lem}\label{vsth} Let graph $H$ be 4-critical and $w \in V(H)$. Then graph $H \odot w$ is always 3-chromatic.
\end{lem}

We are not going to proving this fact here, but delegate to article \cite{ze 12}. We call the feature of 4-critical graph to loose his state of 4-critical graph under vertex split {\em preservation of 4-criticality under vertex split}. Thus, k-criticality is the feature that is preserved under vertex and edge elimination and edge contraction and expectably under vertex split too. It is easy to see that 3-criticality is preserved under vertex split and we ask this same for arbitrary k-critical graphs.

We are going to use lemma \ref{vsth} in the following way. By summation of edge sets of wheels modulo two in order to build new 4-critical graphs we may ignore cases where wheels intersect only in vertices without incident edges, knowing that this can't lead to new 4-critical graphs. Suppose we received 4-critical graph in this way. Then splitting all vertices that were merged by summation backwards we should receive 3-chromatic graph but it might not be true. Let us formulate this fact  as lemma.

\begin{lem}\label{vvsth} Let by summation of edge sets of wheels some wheels intersect in vertices without incident edges. Then resulting graph can't be 4-critical.
\end{lem}

Now we may go over to the main theorem of this chapter.

\begin{thm}\label{thm1} Let $2k+1$  ($k>0$) arbitrary odd wheels be summed in a way that edges of wheels are summed modulo two and all wheels have one overlapping vertex. The resulting graph $M$ is 4-critical if and only if each wheel by summation modulo two looses just two of its edges and resulting graph is planar.
\end{thm}
\begin{proof}
Let us first observe that two wheels may overlap in one or three edges but not in two, thus, to get two annihilating edges a wheel should overlap with two other wheels with one overlapping edge in each.

Let us first assume that wheel graphs are summed observing two edge loss condition. In this case four configurations of new section are possible, see fig. \ref{fig2}. We denote graphs achieved in this way by $w_{k_1\hy k_2\hy ...k_q}$ where $q$ is number of sections, where in each section there are $2k_i+1$ edges, and call them {\em multiwheels}. In this type of denotation we as if ignore three/four ways of section's configuration, but one could easily elaborate denotation with taking these different types of sections into account, see below.

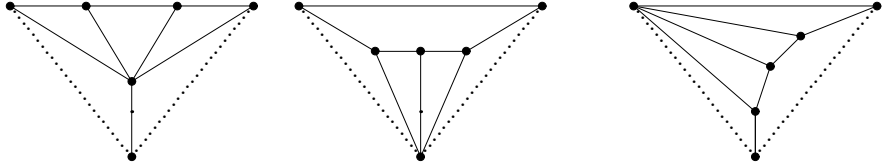
\placedrawing[h]{OCT2.LP}{Four types of sections for planar multiwheels possible. If third type we take in two oriented ways, leftwards and rightwards, then we get one type of section more. When so numbered first two  and two other  sections are mutually dual. It is convenient to characterize type of section by type (rim or spike edge) of lost edges in wheel. Then sections are 1) rim-rim-section (rr-section), 2) spike-spike-section (ss-section), 3) spike-rim-section (sr-section) and 4) rim-spike-section (rs-section). }{fig2}

In figure \ref{fig6}  we see simplest case where summed are three wheels $W_3$ giving multiwheel $w_{1\hy 1\hy 1}$. In place of multiindex $k_1\hy k_2\hy ...k_q$ with hyphens we equally use denotation without hyphens in case no confusion might arise, e.g. $w_{111}$ in place of  $w_{1\hy 1\hy 1}$.

It is easy to see that $w_{111}$ is 4-critical. Indeed, reserving central hub vertex as eventual lonely color vertex (receiving color $D$, see fig. \ref{fig7}) other vertices get forced colors. Further, both hub edge, spike edge and rim edge contractions/eliminations lead to 3-chromatic graph.

Let us assume that the resulting graph is planar. In that case it is convenient to characterize type of section by type (rim or spike edge) of lost edges in wheel. Then sections are 1) rim-rim-section (rr-section), 2) spike-spike-section (ss-section), 3) spike-rim-section (sr-section) and 4) rim-spike-section (rs-section), see fig.\ref{fig2}. Section arisen from $w_1$ we call {\em simple section}, which is of arbitrary type due to automorphisms, i.e., rim edges are spike edges too, and reversely.

Now let us consider first type of section, rr-section, fig.\ref{fig2}. At least one vertex on rim may receive third color and then corresponding section hub receives forth color. Removing rim edge makes possible to color rim with two colors, but removing section's spike edge allows now both previous adjacent vertices color with one color, thus avoiding fourth color.

Let us consider second type section. Now the same applies for the local rim edges of the section. At least one vertex of the local rim should receive third color, and corresponding local rim edge or spike edge elimination may avoid use of fourth color.

Let us consider third (and fourth) type of section. Now outer hinges and vertex adjacent to central hub should receive different colors, but removal at least one edge from section violates this condition and allows to color hinge vertices of inner rim with the same color.

Thus, we have proved that multiwheel is 4-critical.

Let us prove theorem in the other direction.

According lemmas \ref{ov},\ref{us},\ref{vvsth} we are to consider only those sums of wheels where intersections of vertices without incident edges are absent and unclosed sequences of wheels are absent. Even more, if some closed sequences are present, but some wheels as unclosed ends are present, these cases are not to be considered because can't give 4-critical graphs. The only cases are these where only closed sequences of wheels are present. Further, only one closed sequence as cycle is to be considered for further.

Further, let us consider case of non-planar resulting graph, see fig. \ref{fig12}. In that case we have engaged in cycle of wheels some with two spikes that are not sequencing one to other. In these cases such wheel may be as if in two ways taken in cycle along one or other orientation in this wheel between two spikes. It is easy to see that this leads to fact that such wheel cant't give resulting graph as 4-critical, because odd wheel may be divided only in half of odd wheel and half of even wheel. Taking into cyclic path even part of wheel would spoil odd-times-odd structure necessary for 4-critical multiwheel. (See fig.\ref{fig12}, where rights minimal case of non-planar resulting graph is shown to be 3-chromatic.) Thus, we can't afford non-planar spike pairs in edge summation modulo two. We have come to conclusion that closed sequence should be planar.

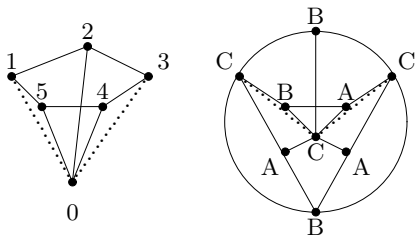
\placedrawing[h]{OCT12.LP}{Illustration to proof. Minimal possible case of nonplanar section for as if eventual nonplanar multiwheel $w_{115}$. Left, we see non-planar section, where annihilated spikes (dotted) are not sequencing ones, and section falls into as if two subwheels, one even subwheel with subrim $1\hy 2\hy 3$ and one odd subwheel with subrim $3\hy 4\hy 5\hy 1$. Right, we color this nonplanar multiwheel into three colors, i.e., it isn't 4-critical, even not 4-chromatic.}{fig12}

Thus, under assumption that all summing wheels have in common one vertex each in each wheel with two incident edges that annihilate by summation modulo two we have proved what was necessary. We have come to 4-critical graph only by specified conditions.

We have proved the theorem.
\end{proof}

\subsection{Denotations for plane multiwheels. Some characteristics}

Let us introduce some notational conventions for wheels and multiwheels. Along with traditional denotation for wheels with capital letter $W$ with index $k$, i.e. $W_k$,  for odd wheels of order $k=2q+1$ we use denotation $w_q$. For plane multiwheels we use letter $w$ with odd ($k\geq 3$) indices, $w_{q_1, ..., q_k}$ where $w_{q_i}$ was $i$-th wheel in summation. Let us introduce quantity $Q=\sum q_i$ and sometimes use denotation $w_Q$ for this multiwheel.

If odd wheel has order $k=2q+1$, it has $n=2q+2$ vertices and $m=2n-2=4q+2$ edges. Similar expressions hold for multiwheels, where in place of $q$ stands $Q$. Indeed, multiwheel has $2Q+1$ vertices and $4Q$ or $2n-2$ edges. To get analogue expressions both for wheels and multiwheels we had to use for wheels in place of quantity $q$ quantity $s=k/2$, giving fractional numbers for odd wheels. It would be interesting to ask then what pre-wheel stands behind $w_{1/2}$. It might be multiedge or edge adjacent to loop.

Both classes, wheels and multiwheels have $m=2n-2$ edges. Let us notice that to the class of graphs $\textsf{G}_{n,2n-2}$ belong these quadrangulating projective plane. In next section we show how this fact turns crucial for multiwheels in generalizing them for projective plane.

Another question would be how to designate indices in multiwheel $w_{q_1, ..., q_k}$ if we wanted to take into account type of sections standing behind corresponding indices. We have four types of sections, therefore we have to equip this index with this additional information. One way would be to use four colors for indices. Other way would be to supply index with diacritic sign, say, $w_{1\hy \hat{3}\hy \check{3}\hy \grave{3}\hy \acute{3}}$ or $w_{1\hat{3}\check{3}\grave{3}\acute{3}}$ for $w_{13333}$ with types of sections in augmented order.

\section{Gr\"{o}tzsch graph, Mycielski's Construction and 4-critical projective plane  multiwheels}

We start with observation that Gr\"{o}tzsch graph \cite{bo 08} may be considered as edge sum modulo two of five wheels $w_1$ and one wheel $w_2$, see fig. \ref{fig9}.

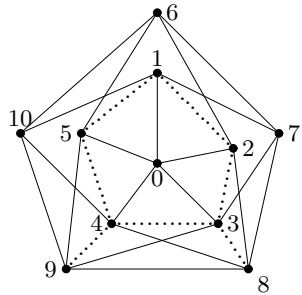
\placedrawing[h]{OCT9.LP}{Example of non-planar graph that is 4-critical: Gr\"{o}tzsch graph. It may be built as edge sum $5 W_3 +W_5$ modulo two. Annihilated edges are depicted as dotted for one section standing for $W_3$ and for $W_5$. }{fig9}

Gr\"{o}tzsch graph is 4-critical and it quadrangulates projective plane, see fig. \ref{fig8}. Indeed, it has $11$ vertices and $20$ edges, i.e., it belongs to $2n-2$-edges-class of graphs, and fig. \ref{fig8} shows how this embedding on projective plane is performed.

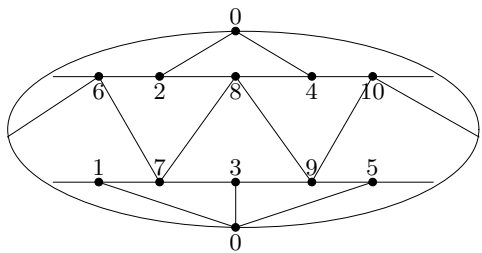
\placedrawing[h]{OCT8.LP}{Gr\"{o}tzsch graph on projective plane. It quadrangulates projective plane. Compare \cite{moh 06}.}{fig8}

If in place of Gr\"{o}tzsch graph formed as $5w_1+w_2$ we take only three plus one wheel, we get graph that is isomorphic to the base graph $w_{1_3}$. Taking this fact into account, we designate this graph $g_{111}$ or $g_{1_3}$ and traditional Gr\"{o}tzsch graph as $g_{11111}$ or $g_{1_5}$. First, let us notice that both graphs are 4-critical, both quadrangulate projective plane, first being planar, bet second - projective planar.  We might ask - are all graphs $q w_1+w_q$ belonging summed according multiwheel summation pattern 4-critical? The answer is quite obviously positive, and we express the fact in the lemma what follows. We say that sum $q w_1+w_q$, $k=2q+1$, $q>0$, modulo two is got according multiwheel pattern if $k$ wheels $w_1$ each looses two edges and $w_q$ looses $k$ edges. This class of graphs we call Gr\"{o}tzsch class.

\begin{lem} For $k=2q+1$, $q>0$, resulting graphs from $k w_1+w_q$ summed according multiwheel pattern are 4-critical and quadrangulate projective plane.
\end{lem}
Of course, for $k=1,2$, we get the base graph and Gr\"{o}tzsch graph, which are 4-critical, and further graphs are 4-critical due to symmetry.

This class $k w_1+w_q$, extending the base graph and the Gr\"{o}tzsch graph may be got as the Mycielski's Construction \cite{bo 76}, page 130. For that we are to take 3-critical graph, i.e., arbitrary odd cycle $C_k$, and apply Mycielski's Construction. Thus we see that if in Mycielski's Construction we replace each new got k-critical graph  with arbitrary k-critical graph then we should receive $k+1$-critical graph, which fact follows from the proof of the Mycielski's Construction's applicability to get k-critical graphs, see \cite{bo 08}. Natural question would arise does there exist Mycielski's Construction's generalization that works backwards too, i.e., that each $k+1$-critical graph has as antecedent k-critical graphs in terms of this or similar construction. In order to include in Mycielski's Construction previous planar class we are to allow to match previous graphs $m$ edges with $m$ new vertices plus extra vertex. This would work for step from 3-critical to 4-critical graphs, and give just our plane class of multiwheels.

Further we are going to build more multiwheels, but the previous class should be the only that were quadrangulating projective plane.

Further we generalize projective planar multiwheels similarly as in case plane multiwheels, i.e., sections of $w_1$ may be replaced with arbitrary odd wheels. Fig. \ref{fig10} shows simplest properly projective plane multiwheel $q_{112}$.

\begin{str}\label{str1}
Let us take odd in number ($k=2q+1$) odd wheels and one wheel $w_q$. Let us take in each of first wheels two proximal spikes and rim edge so that they do not form triangle, and middle spike edge match with central wheel $w_q$, and other two chosen edges (spike and rim edge) match in cyclical sequence of wheels.
\end{str}

The resulting graph built according construction \ref{str1} belongs to $2n-2$ edges class and is 4-critical. We call the resulting graph {\em multiwheel} similarly to those planar ones.

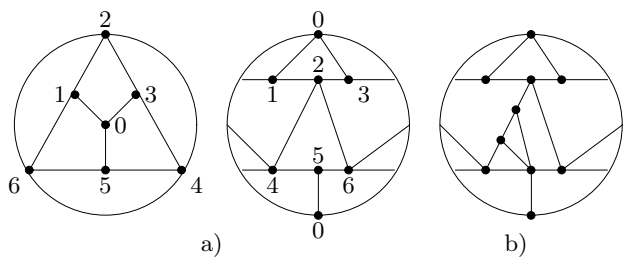
\placedrawing[h]{OCT10.LP}{a) Graphs $w_{111}$ and $g_{111}$ are isomorphic; b)simplest properly projective plane multiwheel with minimal edge number $q_{112}$.}{fig10}

\begin{thm}\label{thm7}
Multiwheels built according construction \ref{str1} are 4-critical.
\end{thm}
\begin{proof}\label{pr7}
Let us use the fact that the base graph belongs to Gr\"{o}tzsch class, and the construction for the base graph extended with non-planar section (see fig.\ref{fig12}) may be used for Gr\"{o}tzsch class in the whole.
\end{proof}

Let us end this section with one more theorem.

\begin{thm}\label{thm8}
Multiwheel quadrangulates projective plane only if it belongs to Gr\"{o}tzsch class.
\end{thm}
\begin{proof}\label{pr8}
The only subclass to be considered is plane multiwheels with simple sections, excluding the base graph, i.e., $w_{1_q}$, $q>1$. It suffices to consider the minimal graph from the class $w_{1_2}$. For graph to quadrangulate surface it is necessary that every edge goes into at least two square cycles. But the edge of $w_{1_5}$ that is incident to vertices of degree three and four doesn't fulfil this condition.
\end{proof}

\section{Octahedral theorem}
Let us formulate what we call octahedral theorem for the plane multiwheels.

As was told in introduction, minor bracket works for the base graph, i.e., $<O^-,O; w_{1_3}>$ is true: $O^- \prec w_{1_3}$, $O \not \prec w_{1_3}$ and $O^- \prec O$. It easily follows from facts that the base graph is only vertex 3-connected, i.e., it has triples of separating vertices, as long as octahedron graph doesn't have. This argument directly applies to plane multiwheels in general, because they are built allowing triples of separating vertices for each section, that excludes possibility for $O$ to be minor. Both plane and projective plane multiwheels have the base graph as their minor. Besides, Gr\"{o}tzsch graph doesn't have $O$ as minor. Indeed, it has 5 cubic vertices, which may be separated with triple of vertices, and adjacent to central hub vertex, and remaining 5 vertices aren't sufficient to hold $O$ as minor. This argument easily generalizes to Gr\"{o}tzsch class in the whole. It only remains to persuade oneself that it works for projective plane multiwheel in general. And again, sections that are differing from simple ones can be separated by triples of vertices, see fig.\ref{fig10}, b. Thus, we have done with the proof.

Let us formulate the fact for arbitrary multiwheels as theorem.

\begin{thm}Minor bracket $<O^-,O>$ works for both plane and projective plane multiwheel graph classes.
\end{thm}


\begin{thebibliography}{99}
\bibitem{bo 08}
Bondy J. A., Murty U.S.R. {\em Graph Theory}, Springer, 2008, 657pp.
\bibitem{bo 76}
Bondy J. A., Murty U.S.R. {\em Graph theory with Applications}, The MacMillian Press Ltd., 1976, 264pp.
\bibitem{jt 95}
Jensen Tommy R., Toft Bjarne. {\em Graph Coloring Problems}, John Willey and Sons, 1995, 295pp.
\bibitem{ko 90}
Koester G. {\em 4-critical 4-valent planar graphs constructed with crowns}, Math Scand.67, 1990, 15-22.
\bibitem{kra 94}
Kratochv\'{\i}l J. {\em About minor closed classes and the generalization of the
notion of free-planar graphs}, personal communication, 1994, 2pp.
\bibitem{moh 06}
Mohar Bojan. {\em Quadrangualtions and 5-critical Graphs on the Projective Plane}, 565-580, in Topics in Discrete Mathematics, Dedicated to Jarik Nesetril on the Occasion of his 60th Birthday, Springer, 2006.
\bibitem{moh 01}
Mohar Bojan, Thomassen Carsten. {\em Graphs on Surfaces}, J. Hopkins Univ. Press, 2001.
\bibitem{ze 98}
Zeps D. {\em Free Minor Closed Classes and the Kuratowski
Theorem}, KAM Series, 98-409, Prague, 1998, 10 pp.
\bibitem{ze 11}
Zeps D. {\em 4-critical wheel graphs of higher order}, \url{arXiv:1106.1336v1}, 2007, 4 pp, \url{congreso.us.es/eurocomb07/program_thursday.html}
\bibitem{ze 08}
Zeps D. {\em Application of the Free Minor Closed Classes in the Context of the Four Color Theorem}, \url{hal-00408145}, 2008, 19pp.
\bibitem{ze 12}
Zeps D. {\em Can graph remain k-critical after vertex split?}, in preparation, 2012.
\end{thebibliography}
\end{document}